\documentclass[11pt,reqno]{amsart}
\usepackage{xifthen,setspace,csquotes}
\usepackage{subfig}
\usepackage{pkgfile}

\title[Fixed points and cycle structure of random permutations]{Fixed points and cycle structure of random permutations}%

\author{Sumit Mukherjee}
\address{Department of Statistics, Columbia University, New York, USA, {\tt sm3949@columbia.edu}}

\begin{document}



\maketitle









\begin{abstract}
Using the recently developed notion of permutation limits
this paper derives the limiting distribution of the number of fixed
points and cycle structure for any convergent sequence of random
permutations, under mild regularity conditions. In particular this
covers random permutations generated from Mallows Model with Kendall's
Tau, $\mu$ random permutations introduced in \cite{HKMRS}, as well as a
class of exponential families introduced in \cite{Mukherjee}.
\end{abstract}


%

\newcommand{\norm}[1]{\left\| #1 \right\|}


\newcommand{\ABS}[1]{\left(#1\right)} 
\newcommand{\veps}{\varepsilon} 




\section{Introduction}

Study of random permutations is an area of classical interest in the
intersection of Combinatorics and Probability theory. Permutation
statistics of interest is indeed a long list which includes number of
fixed points, cycle structure, length of longest increasing
sub-sequence, number of descents, number of cycles, number of
inversions, order of a permutation, etc. Most of this literature
focuses on the case where the permutation $\pi_n$ is chosen uniformly
at random from $S_n$. For example it is well known that the number of
fixed points of a uniformly random permutation converges to $Poi(1)$ in
distribution. More generally, denoting the number of cycles of length $l$ by
$C_n(l)$, we have
\[
\{C_n(1),\cdots,C_n(l)\}\stackrel{d}{\rightarrow}\{
Poi(1),Poi(1/2),\cdots,Poi(1/l)\},
\]
where the limiting Poisson variables are mutually independent. However,
not much is known in this regard outside the realm of the uniform
measure. Possibly the most widely studied non uniform probability
measure on $S_n$ is the Mallows model with Kendall's Tau, first
introduced by Mallows in \cite{Mallows}, which has a p.m.f. of the form
\begin{align}\label{eq:tau}
M_{n,q}(\pi)=\frac{1}{Z_{n,q}}q^{Inv(\pi_n)}.
\end{align}
Here $Inv(\pi_n):=\sum_{1\le i<j\le n}1\{(i-j)(\pi_n(i)-\pi_n(j))<0\}$
is the number of inversions in $\pi_n$, and $q>0$ is a scalar
parameter. In this case an exact formula is known for the normalizing
constant $Z_{n,q}$, and expectation and variance formulas for $Inv(\pi
_n)$ are easy to derive (see for e.g. \cite{Diaconis-Ram}). In \cite
{BDF} Borodin et al.~asked the question of behavior of permutation
statistics such as cycle structure and longest increasing sub-sequence
for general class of Mallows models which includes the Mallows model
with Kendall's Tau. This question was partially answered by
Mueller-Starr in \cite{MS}, where they derived the weak law of the
length of the longest increasing sub-sequence. Specifically, for the
scaling $n(1-q(n))\rightarrow\beta$ they showed that
\[
\frac{1}{\sqrt{n}}LIS(\pi_n)\stackrel{p}{\rightarrow}\cL(\beta),
\]
where $\cL(\beta):=2\beta^{-1/2}\sinh^{-1}(\sqrt{e^{\beta}-1})$ for
$\beta>0$, and $2|\beta|^{-1/2}\sin^{-1}(\sqrt{1-e^\beta})$ for $\beta
<0$. For the scaling when $n(1-q(n))\rightarrow\infty
$, it was shown by Bhatnagar-Peled (\cite{BP}) that
\[
\frac{1}{n\sqrt{1-q(n)}}LIS(\pi_n)\stackrel{p}{\rightarrow}1.
\]
The more recent work of Basu-Bhatnagar (\cite{BB}) consider the case
$q(n)=q\ne1$ is fixed, and prove a weak law for $LIS(\pi_n)$ (they
also derive a central limit theorem for $q<1$). This answers the
question of LIS for Mallows model with Kendall's Tau for all parameter
scalings, at least at the level of weak limits. On the other hand, the
question of the cycle structure still remains largely unanswered. See
however the recent work of Gladkich-Peled, who derive the order of
expected number of cycles in a Mallows random permutation in \cite
[Theorem 1.1]{GP}, when the underlying parameter $q(n)\in(0,1)$ is arbitrary.

In a different direction, in \cite{HKMRS} the authors Hoppen et
al.~proposed a framework where a permutation can be viewed as a
measure. This is described below in brief:

For a permutation $\pi_n\in S_n$ define the measure $\nu_\pi$ on
$[0,1]^2$ as
\[
\nu_{\pi_n}:=\frac{1}{n}\sum_{i=1}^n\delta_{(i/n,\pi_n(i)/n)}.
\]
A sequence of permutations $\{\pi_n\}_{n\ge1}$ with $\pi_n\in S_n$ is
said to converge to a measure $\mu$, if the sequence of probability
measures $\nu_{\pi_n}$ converge weakly to $\mu$. Any such limit is in
$\cM$, the set of probability distribution on the unit square with
uniform marginals.
Any $\mu\in\cM$ is called a permuton (following \cite{HKMRS}), and it
is shown in \cite[Theorem 1.6]{HKMRS} that any $\mu\in\cM$ can indeed
arise as a limit of a sequence of permutations in this manner. See \cite
{Bhattacharya-M,HKMRS} for a more detailed introduction to permutation limits.

If $\{\pi_n\}_{n\ge1}$ is a sequence of random permutations (not
necessarily in the same probability space), the sequence is said to
converge to a deterministic measure $\mu\in\cM$ in probability, if the
sequence of measures $\nu_{\pi_n}$ converge weakly to the measure $\mu$
in probability. Equivalently, for any continuous function $f$ on the
unit square, one has
\[
\lim_{n\rightarrow\infty}\frac{1}{n}\sum_{i=1}^nf\Big(\frac{i}{n},\frac
{\pi_n(i)}{n}\Big)\stackrel{p}{\rightarrow}\int_{[0,1]^2}f(x,y)d\mu.
\]
Using the topology of permutation limits in \cite{Mukherjee} the author
gave a new proof for a large deviation principle (originally proved in
\cite{J}), and used it to analyze a class of exponential families on
the space of permutations. The large deviation principle was re-derived
in \cite{KKRW}, where Kenyon et al.~study permutation ensembles
constrained to have fixed densities of finite number of patterns.


It was shown by Starr in \cite{Starr} that if $\pi_n$ is generated from
a Mallows model with Kendall's Tau with parameter $q(n)$ such that
$n(1-q(n))\rightarrow\beta$, then the sequence of measures $\nu_{\pi
_n}$ converge weakly in probability to a measure ${\mu_{\rho}}_\beta\in
\cM$ induced by the density
\begin{align}\label{eq:mallows_density}
\rho_\beta(x,y):=\frac{(\beta/2)\sinh(\beta/2)}{e^{\beta/4}\cosh(\beta
(x-y)/2)-e^{-\beta/4}\cosh(\beta(x+y-1)/2)},
\end{align}
which is the Frank's Copula (see \cite{Nelson}).
Since $\pi_n$ converges weakly to the measure ${\mu_{\rho}}_\beta$, in
an attempt to understand the marginal distribution of $\pi_n(i)$ one
might conjecture that $\P_n(\pi_n(i)=j)\approx\frac{1}{n} \rho_\beta(i/n,j/n)$.
We will show that this is indeed true, under certain regularity of the
law of the random permutations. We start by introducing some notations.

\begin{defn}
For $l\in[n]:=\{1,2,\cdots,n\}$ let
\[
\cS(n,l):=\{{\bf p}:=(p_1,p_2,\cdots,p_l)\in[n]^l: p_a\ne p_b \text{
for all }a\ne b, a,b\in[l]\}.
\]
Then we have $|\cS(n,l)|={n\choose l}l!$.
For ${\bf p},{\bf q}\in\cS(n,l)$ let $||{\bf p}-{\bf q}||_\infty:=\max
_{a\in[l]}|p_a-q_a|$. Also for ${\bf p}\in\cS(n,l)$ let $\pi_n({\bf
p})$ denote the vector $(\pi_n(p_1),\cdots,\pi_n(p_k))$.
\end{defn}

For every $n\ge1$ let $\pi_n$ be a random permutation on $S_n$ with
law $\P_n$.
In \cite[Def 6.2]{Bhattacharya-M} the authors define a notion of
equi-continuity of random permutations, which they show is implied by
the condition
\begin{align}\label{eq:equi_old}
\lim_{\delta\rightarrow0}\lim_{n\rightarrow\infty}\sup_{{\bf p},{\bf
q},{\bf r}\in\cS(n,l):||{\bf p}-{\bf r}||_\infty\le n\delta}\Big|\frac
{\P_n(\pi_n({\bf p})={\bf q})}{\P_n(\pi_n({\bf r})={\bf q})}-1\Big|=0
\end{align}
(see \cite[Prop 6.2]{Bhattacharya-M}).
In particular for $l=1$ condition \eqref{eq:equi_old} in spirit demands
that the function $\P_n(\pi_n(p)=q)$ is equi-continuous in $p$. In this
paper we will need an extra notion of equi-continuity which demands
that the function $\P_n(\pi_n(p)=q)$ is jointly equi-continuous in
$p,q$. This is stated below:

\begin{defn}
A sequence of random permutations $\pi_n$ is said to be equi-continuous
in both co-ordinates if
\begin{align}\label{eq:equi_new}
\lim_{\delta\rightarrow0}\lim_{n\rightarrow\infty}\sup_{{\bf p},{\bf
q},{\bf r},{\bf s}\in\cS(n,l):||{\bf p}-{\bf r}||_\infty\le n\delta
,||{\bf q}-{\bf s}||\le n\delta}\Big|\frac{\P_n(\pi_n({\bf p})={\bf
q})}{\P_n(\pi_n({\bf r})={\bf s})}-1\Big|=0.
\end{align}
\end{defn}

\begin{defn} Let $\cC$ denote the set of all strictly positive
continuous functions $\rho$ on $[0,1]^2$ with uniform marginals, i.e.
\[
\int_0^1 \rho(x,y)dx=\int_0^1 \rho(x,y)dy=1.
\]
Denote by $\mu_\rho\in\cM$ the measure induced by $\rho$.
\end{defn}

Our first theorem now proves an estimate of $\P_n(\pi_n({\bf p})={\bf
q})$ for vectors ${\bf p},{\bf q}$ if $\pi_n$ is equi-continuous in
both co-ordinates, and converges in the sense of permutation limits to
$\mu_\rho$.
\begin{thm}\label{thm:density_est}
Suppose $\{\pi_n\}_{n\ge1}$ is a sequence of random permutations with
$\pi_n\in S_n$, such that the sequence is equi-continuous in both
co-ordinates, i.e. it satisfies \eqref{eq:equi_new}. If $\{\pi_n\}
_{n\ge1}$ converges to $\mu_\rho$ for some $\rho\in\cC$, we have
\begin{align}\label{eq:density_est}
\lim_{n\rightarrow\infty}\sup_{{\bf p},{\bf q}\in\cS(n,l)}\Big|\frac
{n^l\P_n(\pi_n({\bf p})={\bf q})}{\prod_{a=1}^l \rho\Big(\frac
{p_a}{n},\frac{q_a}{n}\Big)}-1\Big|=0.
\end{align}
\end{thm}

As an immediate corollary of Theorem \ref{thm:density_est} we obtain limiting distribution of the vector $\pi_n({\bf p})$. A more general
version of this corollary was already derived in \cite[Proposition
6.1]{Bhattacharya-M}.

\begin{cor}\label{cor:easy}
Suppose ${\bf p}_n\in\cS(n,l)$ is such that
\[
\lim_{n\rightarrow\infty}\frac{1}{n}{\bf p}_n={\bf x}\in[0,1]^l.
\]
If $\{\pi_n\}_{n\ge1}$ is a sequence of random permutations with $\pi_n\in S_n$ which satisfies \eqref{eq:density_est} for some $\rho
\in\cC$, then
\[
\frac{1}{n}\pi_n({\bf p_n})\stackrel{d}{\rightarrow}\{Y(x_1),\cdots
,Y(x_l)\},
\]
where $\{Y(x_a)\}_{a=1}^l$ are mutually independent with $Y(x_a)$
having the density $\rho(x_a,.)$.
\end{cor}

Having proved Theorem \ref{thm:density_est} we now turn our focus on
the number of fixed points, or more generally the statistic
\[
N_n(\pi_n,\sigma_n):=\sum_{i=1}^n1\{\pi_n(i)=\sigma_n(i)\}
\]
for any $\sigma_n\in S_n$, which
denotes the number of overlaps between $\pi_n$ and $\sigma_n$. In this
notation the number of fixed points of $\pi_n$ equals $N(\pi_n,e_n)$,
where $e_n$ is the identity permutation in $S_n$. By \eqref
{eq:density_est} $N_n(\pi_n,\sigma_n)$ is approximately the sum of $n$
independent variables, and so should be approximately distributed as
Poisson. Our next theorem confirms this conjecture, showing convergence
to Poisson distribution of $N_n(\pi_n,\sigma_n)$ in distribution and in moments.

\begin{thm}\label{thm:fixed_points}
Suppose $\{\pi_n\}_{n\ge1}$ is a sequence of random permutations with $\pi_n\in S_n$ which satisfies
\eqref{eq:density_est} for some $\rho\in\cC$. If $\sigma_n$ converges
to $\mu$, then $\lim_{n\rightarrow\infty}\E N_n(\pi_n,\sigma_n)^k= \E
Poi({\mu[\rho]})^k
$ for any $k\in\N$, where $\mu[\rho]:=\int_{[0,1]^2}\rho(x,y)d\mu$,
and $Poi(\lambda)$ is the Poisson distribution with parameter $\lambda
$. In particular this implies $N_n(\pi_n,\sigma_n)\stackrel
{d}{\rightarrow}Poi({\mu[\rho]})$.

\end{thm}

\begin{remark}
Setting $\sigma_n=e_n$ it follows by Theorem \ref{thm:fixed_points}
that the number of fixed points in $\pi_n$ has a limiting Poisson
distribution with mean $\int_0^1\rho(x,x)dx,$ provided $\{\pi_n\}_{n\ge 1}$
satisfies \eqref{eq:density_est} for some $\rho\in\cC$.

\end{remark}

The random variable $N_n(\pi_n,e_n)=\sum_{i=1}^n 1\{\pi_n(i)=i\}$ is
essentially the number of cycles of length $1$, and a similar intuition
for Poisson approximation holds for cycles of length $l$ for any $l\ge
1$. In order to make this precise, we introduce a few more notations.
\begin{defn}
For any $l\in[n]$ setting $\cU(n,l):=\{{\bf p}\in\cS(n,l):p_1=\min
(p_a,a\in[l])\}$ note that $\cU(n,l)\subset\cS(n,l)$, and $|\cS
(n,l)|=l\times|\cU(n,l)|$.
For ${\bf p}\in\cS(n,l)$ let $T({\bf p})\in\cS(n,l)$ denote the
vector $(p_{2},p_{3},\cdots,p_{l},p_1).$

As an example if $l=3$ and $n=6$ then the vector ${\bf p}=(2,5,4)\in\cU
(n,l)$, as $2=\min(2,5,4)$. In this case $T({\bf p})=(5,4,2)\in\cS
(n,l)$ but does not belong to $\cU(n,l)$, as $5\ne\min(2,5,4)$. Thus
$T$ is the shift operator which shifts every co-ordinate by $1$.

For any $l\ge1$ let
\[
C_n(l):=\sum_{{\bf p}\in\cU(n,l)}1\{\pi_n({\bf p})=T({\bf p})\}=\frac
{1}{l}\sum_{{\bf p}\in\cS(n,l)}1\{\pi_n({\bf p})=T({\bf p})\}.
\]
Then $C_n(l)$ is the number of cycles of length $l$, where the factor
$l$ in the second definition accounts for the fact that every cycle is
counted $l$ times in the second sum. In particular we have
$C_n(1)=N_n(\pi_n,e_n)$ to be the number of fixed points. Also let
\[
c_\rho(l):=\frac{1}{l}\int_{[0,1]^l} \rho(x_1,x_2)\cdots, \rho
(x_l,x_1)dx_1\cdots dx_l.
\]
\end{defn}
The following theorem derives the limiting distribution for $C_{n}(l)$
under condition \eqref{eq:density_est}.
\begin{thm}\label{thm:cycle}
Suppose $\{\pi_n\}_{n\ge 1}$ is a sequence of random permutations with $\pi_n\in S/_n$ which satisfies
\eqref{eq:density_est} for some $\rho\in\cC$. Then for any $\{
k_1,\cdots,k_l\}\in\N^l$ we have
\[
\lim_{n\rightarrow\infty}\E\prod_{a=1}^l C_n(a)^{k_a}= \prod_{a=1}^l
\E Poi({c_\rho(a)})^{k_a}.
\]
In particular this implies
\[
\Big\{C_n(1),\cdots,C_n(l)\}\stackrel{d}{\rightarrow}\Big\{Poi({c_\rho
(1)}),\cdots,Poi({c_\rho(l)})\Big\},
\]
where $\{Poi({c_\rho(i))}\}_{i=1}^l$ are mutually independent.
\end{thm}
\begin{remark}
Thus the number of cycles of length $l$ has a limiting Poisson
distribution with parameter $c_\rho(l)$, whenever the sequence of
permutations $\pi_n$ satisfies \eqref{eq:density_est} for some $\rho\in
\cC$. In particular if $\pi_n$ is uniformly random then \eqref
{eq:density_est} holds for the function $\rho\equiv1$, in which case
$c_\rho(l)=\frac{1}{l}$ for all $l\ge1$. In this case we get back the
classical result that the number of cycles of length $l$ is
asymptotically $Poi(1/l)$, and the random variables $\{C_n(1),\cdots
,C_n(l)\}$ are mutually asymptotically independent for any $l\in\N$.
\end{remark}

\subsection{Applications}

As applications of Theorem \ref{thm:fixed_points} and Theorem \ref
{thm:cycle}, we will now derive the limit distributions of the number
of fixed points and cycle structures for three classes of non uniform
distributions on $S_n$.

\begin{enumerate}
\item[(i)]
The first result in this direction is the next corollary, which deals
with the Mallows model with Kendall's Tau.

\begin{cor}\label{cor:mallows}
Suppose $\pi_n$ is a random permutation on $S_n$ generated from the
Mallows model with Kendall's Tau defined in \eqref{eq:tau}, such that
$n(1-q(n))\rightarrow\beta\in(-\infty,\infty)$. In this case the
following conclusions hold with $\rho_\beta$ as defined in \eqref
{eq:mallows_density}.

\begin{enumerate}
\item
If $\{\sigma_n\}_{n\ge1}$ is a sequence of non random permutations
with $\sigma_n\in S_n$ converging to $\mu$, then $N_n(\pi_n,\sigma_n)$
converges to $Poi({\mu[\rho_\beta]})$ in distribution and in moments.

\item
$\Big\{C_n(1),\cdots,C_n(l)\}$ converges to $\Big\{Poi({c_{\rho_\beta
}(1)}),\cdots,Poi({c_{\rho_\beta}(l)})\Big\}$ in distribution and in
moments, where $\{Poi({c_{\rho_\beta}(i))}\}_{i=1}^l$ are mutually independent.

\end{enumerate}

\end{cor}

As an illustration of the Poisson approximation, in figure \ref{fig:degree_sequence} we compare the
histogram of the number of fixed points in a permutation of size
$n=100$ with the limiting Poisson prediction. We used $10000$
independent observations from the Mallows model with Kendall's Tau with
parameter $q(n)=e^{-20/n}$. From the picture it seems that the Poisson
prediction is fairly accurate for $n=100$. Since $q(n)<1$ it is
expected that this model will have more fixed points than a uniformly
random permutation, which is reflected in the fact that the mean of the
Poisson distribution is much larger than 1.

\begin{figure*}[h]
\centering
\begin{minipage}[c]{1.0\textwidth}
\centering
\includegraphics[width=5in]
    {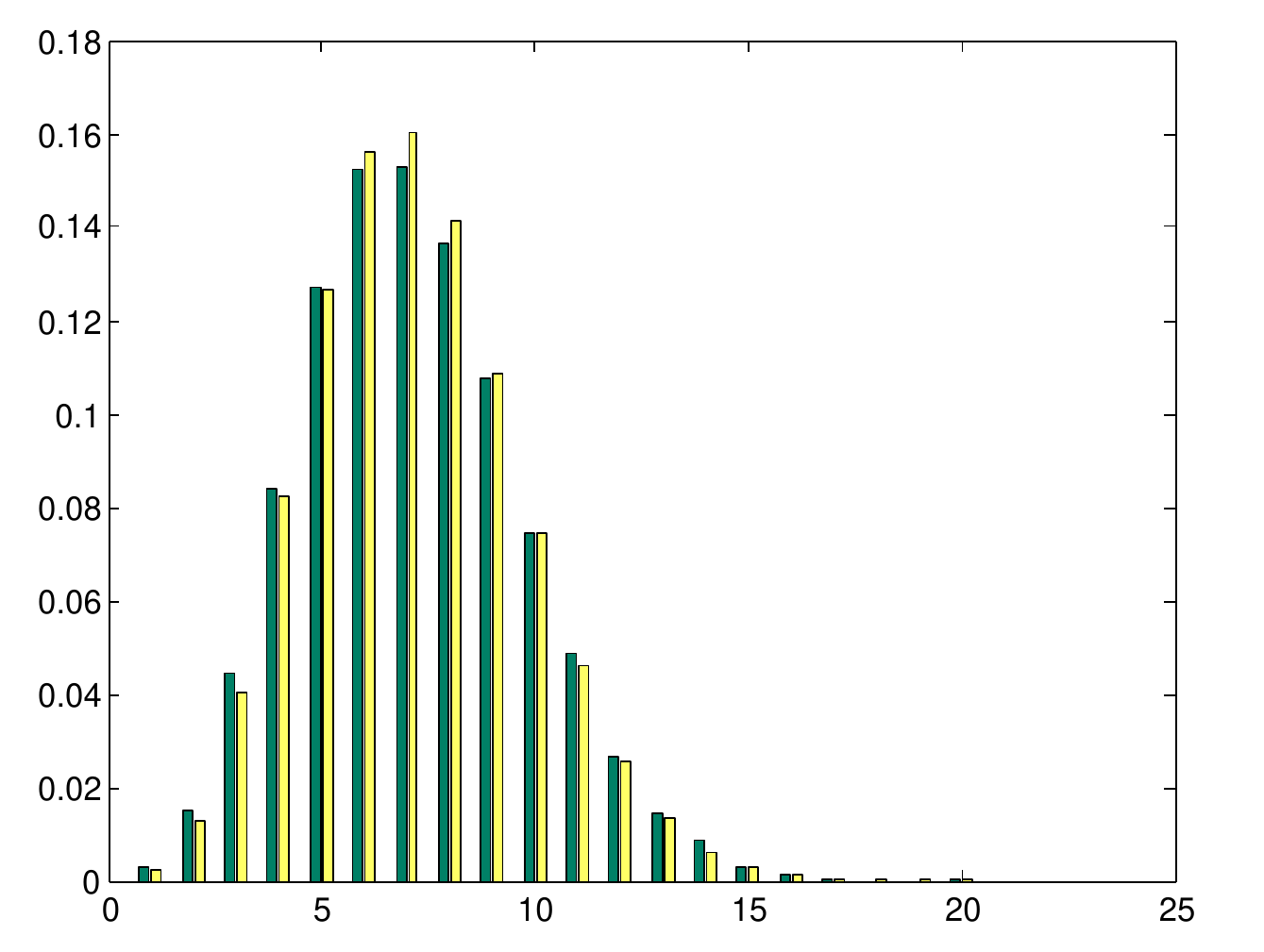}\\
\end{minipage}
\caption{\small{Bar plot of empirical distribution (from 10000 observations) of number of fixed points in a permutation of size $n=100$ from the Mallows model with Kendall's Tau with parameter $q_n=e^{-20/n}$ in green, compared to  the Poisson prediction in yellow.}}
\label{fig:degree_sequence}
\end{figure*}

\item[(ii)]
Another class of non uniform measures on permutations introduced by
the author in \cite{Mukherjee} is the following:

For any continuous function $f$ on the unit square, let $\Q_{n,\theta}$
be a one parameter exponential family with sufficient statistic
\[
\sum_{i=1}^n f\Big(\frac{i}{n},\frac{\pi_n(i)}{n}\Big)=n\nu_{\pi_n}[f].
\]
More precisely, the p.m.f. is given by
\begin{align}\label{eq:new_class}
\Q_{n,\theta,f}(\pi)=e^{n\theta\nu_{\pi_n}[f]-Z_n(f,\theta)},
\end{align}
where $Z_n(f,\theta)$ is the log normalizing constant of the model. In
particular the permutation model obtained by the following two specific
choices have been studied in the Statistics literature:

\begin{enumerate}
\item
$f(x,y)=|x-y|$, which gives the statistic $\sum_{i=1}^n|\pi(i)-i|$
known as the Spearman's Footrule.

\item
$f(x,y)=(x-y)^2$, which gives the statistic $\sum_{i=1}^n(\pi(i)-i)^2$
known as Spearman's rank correlation Statistic.
\end{enumerate}

See \cite[Chapter 5,6]{Diaconis} for more on these and other non
uniform permutation models considered in the Statistics literature. The
convergence of a sequence of random permutations $\pi_n$ generated from
$\Q_{n,f,\theta}$ of \eqref{eq:new_class} was shown in \cite[Theorem
1.4]{Mukherjee}. Building on this result, the next corollary derives
the limiting distributions of the number of fixed points and cycle
structure for a permutation $\pi_n$ generated from this model.

\begin{cor}\label{cor:new_limit}
Suppose $\pi_n$ is a random permutation on $S_n$ generated from the
model $\Q_{n,f,\theta}$ defined in \eqref{eq:new_class} for some
function $f$ which is continuous on the unit square. In this case the
following conclusions hold:
\begin{enumerate}
\item
The sequence $\{\pi_n\}_{n\ge1}$ converges weakly to a non random
measure $\mu_{f,\theta}\in\cM$ with a continuous density $g_{f,\theta}(.,.)$.

\item
If $\{\sigma_n\}_{n\ge1}$ is a sequence of non random permutations
with $\sigma_n\in S_n$ converging to $\mu$, then $N_n(\pi_n,\sigma_n)$
converges to $Poi({\mu[g_{f,\theta}]})$ in distribution and in moments.

\item
$\Big\{C_n(1),\cdots,C_n(l)\}$ converges to $\Big\{Poi({c_{g_{f,\theta
}}(1)}),\cdots,Poi({c_{g_{f,\theta}}(l)})\Big\}$ in distribution and in
moments, where $\{Poi({c_{g_{f,\theta}}(i)})\}_{i=1}^l$ are mutually
independent.

\end{enumerate}

\end{cor}

\item[(iii)]
The final class of permutation models that we consider is a non
parametric model with a measure as the parameter, as opposed to the
previous two models which are one parameter models. This class of
models will be referred to as $\mu$ random permutations, and was first
introduced in \cite{HKMRS}.

Given any $\mu\in\cM$ let $(X_1,Y_1),\cdots,(X_n,Y_n)$ be i.i.d.
random vectors with law $\mu$. Define a permutation $\pi_n^{\mu}\in
S_n$ as follows:

If there exists $l\in[n]$ such that $X_l=X_{(i)},Y_l=Y_{(j)}$, then
set $\pi_n^{\mu}(i)=j$.
To visualize this definition differently, let $\sigma_x$ and $\sigma_y$
be the permutations of order $n$ such that $x_{\sigma_x(1)}< x_{\sigma
_x(2)}<\cdots x_{\sigma_x(n)}$ and $y_{\sigma_y(1)}< y_{\sigma
_y(2)}<\cdots y_{\sigma_y(n)}$, respectively (since the marginals of
$\mu$ are uniform, ties do not occur with probability 1). Then the
above definiton is equivalent to setting $\pi^{\mu}_n=\sigma_y^{-1}\circ
\sigma_x$. It is easy to see that if $\mu$ has density $\rho$, then for
any permutation $\pi_n\in S_n$ one has
\begin{align}\label{eq:mu_random}
\P_n(\pi_n^\mu=\pi_n)=n!\int_{0<u_1<\cdots<u_n<1,0<v_1<\cdots<v_n<1}
\prod_{i=1}^n \rho\Big(u_i,v_{\pi_n(i)}\Big)du_i dv_i.
\end{align}
By \cite[Lemma 4.2]{HKMRS} it follows that $\pi_n^\mu$ converges weakly
to $\mu$ in probability.

Our next corollary derives limiting distributions for $\mu$ random
permutations, when the measure $\mu$ has a continuous density function
with respect to Lebesgue measure.

\begin{cor}\label{cor:mu_random}
Suppose $\pi_n$ is a $\mu_\rho$ random permutation in $S_n$ for some
$\rho\in\cC$.
In this case the following conclusions hold:
\begin{enumerate}
\item
If $\{\sigma_n\}_{n\ge1}$ is a sequence of non random permutations
with $\sigma_n\in S_n$ which converges to $\mu$, then
$N_n(\pi_n,\sigma_n)$ converges to $Poi({\mu[\rho]})$ in distribution
and in moments.

\item
$\Big\{C_n(1),\cdots,C_n(l)\}$ converges to $\Big\{Poi({c_{\rho
}(1)}),\cdots,Poi({c_{\rho}(l)})\Big\}$ in distribution and in moments,
where $\{Poi({c_{\rho}(i)})\}_{i=1}^l$ are mutually independent.
\end{enumerate}
\end{cor}
\end{enumerate}

Even though the weak convergence of the random permutation sequence is
the main ingredient in all the above results, the equi-continuity in
both co-ordinates is not just a technical requirement. The following
example shows that the conclusions of Theorems \ref{thm:density_est}
and \ref{thm:fixed_points} might not hold if the equi-continuity
condition fails.

\begin{ppn}\label{ppn}
Let $\R_{n,\theta}$ be a probability distributon on $S_n$ with the
p.m.f.
\[
\R_{n,\theta}(\pi_n)=e^{\theta N_n(\pi_n,e_n)-Z_n(\theta)}
\]
where $e_n$ is the identity permutation, and $N_n(\pi_n,e_n)$ is the
number of fixed points in $\pi_n$. Then for every $\theta\ne0$ the
following conclusions hold:

\begin{enumerate}
\item[(a)]
The random variable $N_n(\pi_n,e_n)$ converges to a Poisson random
variable with mean $e^\theta$ in distribution and in moments.

\item[(b)]
$\pi_n$ converges weakly to $u$, the uniform distribution on $[0,1]^2$
which is free of $\theta$.

\item[(c)]
\[
\frac{\R_{n,\theta}(\pi_n(1)=1,\pi_n(2)=2)}{\R_{n,\theta}(\pi_n(1)=2,\pi
_n(2)=1)}=e^{2\theta}\ne1.
\]

\end{enumerate}

\end{ppn}

\begin{remark}
Thus even though the sequence of random permutations under $\R_{n,\theta
}$ converge to Lebesgue measure (which is free of $\theta$ and has a
continuous density), the number of fixed points has a limiting Poisson
distribution which depends on $\theta$. This is the case as
equi-continuity in both coordinates does not hold here, as demonstrated
by part (c) of the proposition.

\end{remark}
\subsection{Scope of future research}
For the Mallows model with Kendall's Tau, the results of this paper
only apply for the case $n(1-q(n))=O(1)$. If $n(1-q(n))\rightarrow\infty
$, one should expect the number of fixed points to go to $+\infty$, and
computing the weak limits/limiting distribution after centering/scaling
in this case remain open. In another vein, one might expect that
convergence in the sense of permutations along with {``mild''}
regularity conditions imply the weak convergence of $LIS$, as worked
out for the Mallows model with Kendall's Tau in \cite{MS}. Finally,
computing the limiting density for the model defined in \eqref
{eq:new_class} might help give a more explicit description for the
parameters of the limiting distributions of Corollary \ref
{cor:new_limit}, as well as give non trivial copulas (bivariate
distributions with uniform marginals) which constitute a subject area
of its own in Finance.

\subsection{Outline of the paper}
Section \ref{sec:two} gives the proof of Theorem \ref{thm:density_est},
Corollary \ref{cor:easy}, and Theorems \ref{thm:fixed_points} and \ref
{thm:cycle}. Section \ref{sec:three} concludes the paper by proving
Corollaries \ref{cor:mallows}-\ref{cor:mu_random}, and Proposition \ref{ppn}.

\section{Proofs of main results}\label{sec:two}
\subsection{Proof of Theorem \ref{thm:density_est} and Corollary \ref{cor:easy}}

\begin{proof}[Proof of Theorem \ref{thm:density_est}]

For $k\in\N$ setting
\[
\epsilon_n(k):=\sup_{{\bf p},{\bf q},{\bf r},{\bf s}\in\cS(n,l):||{\bf
p}-{\bf r}||_\infty\le n/k}\Big|\frac{\P_n(\pi_n({\bf p})={\bf q})}{\P
_n(\pi_n({\bf r})={\bf s})}-1\Big|,
\]
condition \eqref{eq:equi_new} can be stated as
\begin{align}\label{eq:obvious}
\lim_{k\rightarrow\infty}\lim_{n\rightarrow\infty}\epsilon_n(k)=0.
\end{align}
Fix $k\in\N$ and partition $(0,1]$ as $\cup_{a=1}^kI_i$ with $I_a:=\Big
(\frac{i-1}{k},\frac{i}{k}\Big]$. Setting
\[
A_n^k:=\prod_{a=1}^lI_{\lceil kp_a/n\rceil},\quad B_n^k:=\prod
_{a=1}^lI_{\lceil kq_a/n\rceil}
\]
note that $\frac{1}{n}{\bf p}\in A_n^k, \frac{1}{n}{\bf q}\in B_n^k$.
Now for any ${\bf r},{\bf s}\in\cS(n,l)$ such that $\frac{1}{n}{\bf
r}\in A_n^k,\frac{1}{n}{\bf s}\in B_n^k$ we have
\begin{align*}
\Big|\frac{\P_n(\pi_n({\bf p})={\bf q})}{\P_n(\pi_n({\bf r})={\bf
s})}-1\Big|\le\epsilon_n(k),
\end{align*}
which on summing over ${\bf r}\in A_n^k,{\bf s}\in B_n^k$ and noting
that the number of terms summed is at least $(n-1)^{2l} k^{-2l}$ gives
\begin{align*}
\P_n(\pi_n({\bf p})={\bf q})\le&(1+\epsilon_n(k))
\frac{k^{2l}}{(n-1)^{2l}}\sum_{{\bf r}\in A_n^k, {\bf s}\in B_n^k} \P
_n(\pi_n({\bf r})={\bf s})\\
=&(1+\epsilon_n(k))
\frac{k^{2l}n^l}{(n-1)^{2l}}\E{\nu_\pi}_n^{(l)}[A_n^k\times B_n^k]\\
=&(1+\epsilon_n(k))
\frac{k^{2l}n^l}{(n-1)^{2l}}\E\prod_{a=1}^l\nu_{\pi_n}[I_{\lfloor k
p_a/n\rfloor}\times I_{\lfloor k q_a/n\rfloor}],
\end{align*}
where $\nu^l_{\pi_n}$ denotes the $l$ fold product measure of $\nu_{\pi
_n}$. Using the fact that $\rho$ is the density for $\mu_\rho$ this
readily gives
\begin{align}
\sup_{{\bf p},{\bf q}\in\cS(n,l)}\frac{n^l \P_n(\pi_n({\bf p})={\bf
q})}{\prod_{a=1}^l\rho(\frac{p_a}{n},\frac{q_a}{n})}\le&(1+\epsilon
_n(k))\times\frac{n^{2l}}{(n-1)^{2l}} \label{eq:bound-1}\\
\times& \E\sup_{{\bf i},{\bf j}\in[k]^l} \frac{ \prod_{a=1}^l\nu_{\pi
_n}[I_{i_a}\times I_{j_a}]}{\prod_{i=1}^l\mu_\rho[I_{i_a}\times
I_{j_a}]}\label{eq:bound-2}\\
&\times\sup_{{\bf x},{\bf y},{\bf z},{\bf w}\in[0,1]^l:||{\bf x}-{\bf
z}||_\infty\le1/k, ||{\bf y}-{\bf w}||_\infty\le1/k}\frac{\prod
_{a=1}^l\rho(x_a,y_a)}{\prod_{a=1}^l\rho(z_a,w_a)}.\label{eq:bound-3}
\end{align}
The term in the r.h.s. of \eqref{eq:bound-1} converges to $1$ on
letting $n\rightarrow\infty$ followed by $k\rightarrow\infty$, using
\eqref{eq:obvious}.

Since $\nu_{\pi_n}$ converges to $\mu_\rho$, by \cite[Theorem
5.2]{HKMRS} we have
\begin{align*}\max_{a\in[l]}\max_{i_a\in[k]}\Big|\nu_{\pi
_n}[I_{i_a}\times I_{j_a}]-\mu_\rho[I_{i_a}\times I_{j_a}]\Big|\stackrel
{p}{\rightarrow}0,
\end{align*}
which along with the observation that $\mu_\rho[I_{i_a}\times I_{j_a}]$
is uniformly bounded away from $0$ gives
\begin{align*}
\max_{a\in[l]}\max_{i_a\in[k]} \frac{ \prod_{a=1}^l\nu_{\pi
_n}[I_{i_a}\times I_{j_a}]}{\prod_{a=1}^l\mu_\rho[I_{i_a}\times
I_{j_a}]}\stackrel{p}{\rightarrow}1
\end{align*}
as $n\rightarrow\infty$, for $k$ fixed.
An application of Dominated Convergence theorem implies that the term
in \eqref{eq:bound-2} converges to $1$ as well. Finally \eqref
{eq:bound-3} is free of $n$, and converges to $1$ as $k\rightarrow\infty
$ by continuity of $\rho$. Combining this gives
\[
\limsup_{k\rightarrow\infty}\limsup_{n\rightarrow\infty}\sup_{{\bf
p},{\bf q},{\bf r}\in\cS(n,l):||{\bf p}-{\bf r}||_\infty\le n/k} \frac
{n^l \P_n(\pi_n({\bf p})={\bf q})}{\prod_{a=1}^l\rho(\frac{p_a}{n},\frac
{q_a}{n})}\le1,
\]
thus giving the upper bound in \eqref{eq:density_est}. A similar proof
gives the lower bound, thus completing the proof of the theorem.
\end{proof}

We now introduce some auxiliary variables, to be used in the proofs of
Corollary \ref{cor:easy}, and Theorems \ref{thm:fixed_points} and \ref
{thm:cycle}.

\begin{defn}\label{def:stein}
For every $n\ge1$ let $\{Z_n(1),\cdots,Z_n(n)\}$ be mutually
independent random variables supported on $[n]$ such that the marginal
laws are given by
\[
\Q_n(Z_n(p)=q)=\frac{\rho(p/n,q/n)}{\sum_{s=1}^n \rho(p/n,
q/n)}
\]
for some $\rho\in\cC$.
Also set
\[
M_n(\sigma_n):=\sum_{p=1}^n1\{Z_n(p)=\sigma_n(p)\},
\]
and for $l\ge1$ set
\[
D_n(l):=\sum_{{\bf p}\in\cU(n,l)}1\{Z_n({\bf p})=T({\bf p})\}.
\]
\end{defn}

\begin{proof}[Proof of Corollary \ref{cor:easy}]
With $Z_n$ as constructed in definition \ref{def:stein} we have
\begin{align*}
&\sum_{{\bf q}\in\cS(n,l)}\Big|\P_n(\pi_n({\bf p}_n)={\bf q})-\Q
_n(Z_n({\bf p}_n)={\bf q})\Big|\\
\le&\max_{{\bf p},{\bf q}\in\cS(n,l)}\Big|\frac{\P_n(\pi_n({\bf
p}_n)={\bf q})}{\Q_n(Z({\bf p}_n)={\bf q})}-1\Big|\sum_{{\bf q}\in \cS
(n,l)}\Q_n(Z_n({\bf p}_n)={\bf q})\\
=&\max_{{\bf p},{\bf q}\in\cS(n,l)}\Big|\frac{\P_n(\pi_n({\bf
p}_n)={\bf q})}{\Q_n(Z_n({\bf p}_n)={\bf q})}-1\Big|,
\end{align*}
which goes to $0$ by \eqref{eq:density_est}. This implies that the laws
of $\pi_n({\bf p}_n)$ and $Z_n({\bf p}_n)$ are close in total
variation. Since the desired conclusion can be verified easily for
$Z_n({\bf p}_n)$, the proof is complete.
\end{proof}

\subsection{Proofs of Theorem \ref{thm:fixed_points} and \ref{thm:cycle}}

We will use Stein's method based on dependency graphs to prove Poisson
limit theorems, as explained below:

Let $\{X_\alpha\}_{\alpha\in I}$ be a finite set of Bernoulli random
variables. A dependency graph for $\{X_\alpha\}_{\alpha\in I}$ is a
graph with node set $I$ and edge set $E$, such that if $I_1,I_2$ are
disjoint subsets of $I$ with no edges connecting them, then $\{X_\alpha
\}_{\alpha\in I_1}$ and $\{X_\beta\}_{\beta\in I_2}$ are independent.
Let $N_\alpha$ be the neighborhood of vertex $\alpha$, i.e. $N(\alpha
):=\{\beta\in I:(\alpha,\beta)\in E\}\cup\{\alpha\}$. Then one has the
following Poisson approximation result, first proved in \cite{AGG}.

\begin{thm}\label{thm:stein}\cite[Theorem 15]{CDM}
Let $\{X_\alpha\}_{\alpha\in I}$ be a finite set of Bernoulli random
variables with dependency graph $(I,E)$. 
Then setting $\lambda:=\sum_{\alpha\in I} p_\alpha$, $W:=\sum_{\alpha
\in I} X_\alpha$ we have
\[
||\cL(W)-\cL(Poi(\lambda))||_{TV}\le\sum_{\alpha\in I}\sum_{\beta\in
N(\alpha)/\{\alpha\}}\E X_{\alpha}X_\beta+\sum_{\alpha\in I}\sum_{\beta
\in N(\alpha)}\E X_\alpha\E X_\beta.
\]
\end{thm}

The following lemma uses Theorem \ref{thm:stein} to prove two Poisson
limits which will be used in the proofs of Theorems \ref
{thm:fixed_points} and \ref{thm:cycle}.

\begin{lem}\label{lem:stein}

\begin{enumerate}
Let $M_n(\sigma_n)$ and $D_n(l)$ be as in definition \ref{def:stein}.
\item[(a)]
If $\sigma_n$ converges to $\mu\in\cM$ in the sense of permutation
limits, then we have $M_n(\sigma_n)\stackrel{d}{\rightarrow}P_{\mu[\rho
]}$, and
\[
\lim_{n\rightarrow\infty}\E M_n(\sigma_n)^k=\E Poi({\mu[\rho]})^k,\text
{ for all } k\in\N.
\]

\item[(b)]
For any $l\in\N$ we have $D_n(l)\stackrel{d}{\rightarrow}P_{c_\rho(l)}$, and
\[
\lim_{n\rightarrow\infty}\E D_n(l)^k=\E Poi({c_\rho(l)})^k,\text{ for all }
k\in\N.
\]

\end{enumerate}

\end{lem}

\begin{proof}
Setting $m:=\inf_{0\le x,y\le1}\rho(x,y)$, $M:=\sup_{0\le x,y\le1}\rho
(x,y)$ we have $0<m\le M<\infty$.
\begin{enumerate}
\item[(a)]
Since the random variables $X_p=1\{Z_n(p)=\sigma_n(p)\}$ for
$p=1,2,\cdots,n$ are mutually independent, the dependency graph of $\{
X_1,X_2,\cdots,X_n\}$ is empty. It then follows by Theorem \ref
{thm:stein} that
\[
||\cL(M_n(\sigma_n))-\cL(Poi({\lambda_n}))||_{TV}\le\sum_{p=1}^n\Big
[\frac{\rho(p/n,\sigma_n(p)/n)}{\sum_{q=1}^n \rho(p/n,q/n)}\Big]^2\le
\frac{1}{n}\times\frac{M^2}{m^2}
\]
where
\[
\lambda_n=\sum_{p=1}^n \frac{\rho(p/n,\sigma_n(p)/n)}{\sum_{q=1}^n\rho
(p/n,q/n)}\stackrel{n\rightarrow\infty}{\rightarrow}\int_{[0,1]^2}\rho
(x,y)d\mu=\mu[\rho],
\]
and so $M_n(\sigma_n)$ converges to $Poi({\mu[\rho]})$ in distribution.
To conclude convergence in moments it suffices to show that
$\limsup_{n\rightarrow\infty}\E M_n(\sigma_n)^k<\infty$ for every $k\in
\N$.
To see this, set
\[
\widetilde{\cS}(n,l):=\{{\bf p}\in\cS(n,l):p_1<p_2<\cdots<p_l\}
\]
denote the set of all $n$ tuples in increasing order, and note that
\begin{align*}
\E M_n(\sigma_n)^k=\sum_{{\bf p}\in[n]^k}\Q_n(Z_n({\bf p})=\sigma
_n({\bf p}))\le&\sum_{l=1}^kk^l\sum_{{\bf p}\in\widetilde{\cS}(n,l)}\Q
_n(Z_n({\bf p})=\sigma_n({\bf p})).
\end{align*}
Here the factor $k^l$ in the r.h.s. above accounts for the fact that
a specific term $\{Z_n({\bf p})=\sigma_n({\bf p})\}$ with ${\bf p}\in
\widetilde{\cS}(n,l)$ can arise from at most $k^l$ terms in $[n]^k$.
Since $|\widetilde{\cS}(n,l)|={n\choose l}$, we can bound the r.h.s.
above by
\begin{align*}
\sum_{l=1}^kk^l\sum_{{\bf p}\in\widetilde{\cS}(n,l)}\prod_{a=1}^l \frac
{\rho(p_a/n,\sigma_n(p_a)/n)}{\sum_{q_a=1}^n\rho(p_a/n,q_a/n)}
\le\sum_{l=1}^k\frac{k^l}{l!}\Big(\frac{M}{m}\Big)^l<\infty.
\end{align*}

\item[(b)]
The proof of part (b) is similar to the proof of part (a). For ${\bf
p}\in\cU(n,l)$ setting $X_{{\bf p}}=1\{Z_n({\bf p})=T({\bf p})\}$ note
that $X_{{\bf p}}$ is independent of $X_{{\bf q}}$ whenever the indices
${\bf p}$ and ${\bf q}$ have no overlap. Thus the dependency graph of
the random variables $\{X_{{\bf p}},{\bf p}\in\cU_{n,l}\}$ has maximum
degree at most ${n-1\choose l-1}l!$. Also for any ${\bf p},{\bf q}$
which overlap we have $\E X_{\bf p}X_{\bf q}=0$ unless ${\bf p}={\bf
q}$. Thus an application of Theorem \ref{thm:stein} gives
\[
||\cL(D_n(l))-\cL(P_{\lambda_n})||\le{n\choose l}(l-1)!\times
{n-1\choose l-1}l!\times\frac{M^{2l}}{n^{2l}m^{2l}}\le\frac
{1}{n}\times\frac{M^{2l}}{m^{2l}},
\]
with
\[
\lambda_n=\frac{1}{l}\sum_{{\bf p}\in\cS_{n,l}}\frac{\rho
(p_1/n,p_2/n)}{\sum_{q_1=1}^n\rho(p_1/n,q_1/n)}\times\cdots\times\frac
{\rho(p_l/n,p_1/n)}{\sum_{q_l=1}^n\rho(p_l/n,q_l/n)}\stackrel
{n\rightarrow\infty}{\rightarrow}c(l),
\]
and so $D_n(l)$ converges to $Poi({c(l)})$ in distribution. Convergence
in moments follows by a similar calculation as before.\qedhere
\end{enumerate}
\end{proof}

\begin{proof}[Proof of Theorem \ref{thm:fixed_points}]

Let $\{Z_n(1),\cdots,Z_n(n)\}$ and $M_n(\sigma_n)$ be as defined in \ref
{def:stein}. Then using part (a) of Lemma \ref{lem:stein} and the fact
that the Poisson distribution is characterized by its moments, it
suffices to show that for every $k\in\N$ we have
\[
\lim_{n\rightarrow\infty}|\E N_n(\pi_n,\sigma_n)^k-\E M_n(\sigma_n)^k|=0.
\]
To this effect setting $Z_n({\bf p})=(Z_n(p_1),\cdots,Z_n(p_k))$ for
${\bf p}\in[n]^k$ we have
\begin{align*}
|\E N_n(\pi_n,\sigma_n)^k-\E M_n(\sigma_n)^k|
\le\sum_{{\bf p}\in[n]^k}\Big|\Big\{\P_n\Big(\pi_n({\bf p})=
\sigma_n({\bf p})\Big)-\Q_n\Big(Z_n({\bf p})=\sigma_n({\bf p})\Big)\Big
\}\Big|
\end{align*}
First note that the events $\{\pi_n({\bf p})=\sigma_n({\bf p})\}$ and
$\{Z_n({\bf p})=\sigma_n({\bf p})\}$ have positive probability for all
${\bf p}\in[n]^k$, and so for any ${\bf p}\in[n]^k$ setting $L=L({\bf
p})$ denote the number of distinct indices gives the bound

\begin{align*}
&\Big|\Big\{\P_n\Big(\pi_n({\bf p})=
\sigma_n({\bf p})\Big)-\Q_n\Big(Z_n({\bf p})=\sigma_n({\bf p})\Big)\Big
\}\Big|\\
\le&\max_{{\bf p},{\bf q}\in\cS(n,L)}\Big|\frac{\P_n(\pi_n({\bf
p})={\bf q})}{\Q_n(Z_n({\bf p})={\bf q})}-1\Big|\Q_n\Big(Z_n({\bf
p})=\sigma_n({\bf p})\Big).
\end{align*}
Since $L({\bf p})\le k$, taking a maximum over $L$ and summing over
${\bf p}\in[n]^k$ gives the bound
\begin{align}
|\E N_n(\pi_n,\sigma_n)^k-\E M_n(\sigma_n)^k|\le \left\{\max_{l\in
[k]}\max_{{\bf p},{\bf q}\in\cS(n,l)}\Big|\frac{\P_n(\pi_n({\bf
p})={\bf q})}{\Q_n(Z_n({\bf p})={\bf q})}-1\Big|\right\}\E M_n(\sigma_n)^k.
\label{eq:fixed_points_bound}
\end{align}
By \eqref{eq:density_est} we have
\[
\max_{l\in[k]}\max_{{\bf p},{\bf q}\in\cS(n,l)}\Big|\frac{\P_n(\pi
_n({\bf p})={\bf q})}{\Q_n(Z_n({\bf p})={\bf q})}-1\Big|\rightarrow0.
\]
Since Lemma \ref{lem:stein} implies
\begin{align*}
\limsup_{n\rightarrow\infty}\E M_n(\sigma_n)^k=\E Poi(\mu[\rho])^k<\infty,
\end{align*}
the r.h.s. of \eqref{eq:fixed_points_bound} converges to $0$ as
$n\rightarrow\infty$, thus completing the proof of the theorem.
\end{proof}

\begin{proof}[Proof of Theorem \ref{thm:cycle}]
Let $\{Z_n(1),\cdots,Z_n(n)\}$ and $\{D_n(a),1\le a\le l\} $ be as
defined in \ref{def:stein}. Then by part (b) of Lemma \ref{lem:stein},
for any finite collection of non negative integers $k_1,k_2,\cdots,k_l$
we have
\[
\lim_{n\rightarrow\infty}\prod_{a=1}^l\E{D}_n(a)^{k_a}=\prod_{a=1}^l\E
Poi({c_\rho(a)})^{k_a}.
\]
Thus to complete the proof it suffices to show the following:
\begin{align}\label{eq:cycles_1}
\lim_{n\rightarrow\infty}\left|\E\prod_{a=1}^lD_n(a)^{k_a}- \prod
_{a=1}^l\E D_n(a)^{k_a}\right|=0,\\
\label{eq:cycles_2}
\lim_{n\rightarrow\infty}\left|\E\prod_{a=1}^l C_n(a)^{k_a}-\E\prod
_{a=1}^l D_n(a)^{k_a}\right|=0.
\end{align}
%
%

For showing \eqref{eq:cycles_1} we have
\begin{align}
\notag|\E\prod_{a=1}^lD_n(a)^{k_a}- \prod_{a=1}^l\E{D}_n(a)^{k_a}|\le
&\sum_{\Gamma}\Big|\Q_n\Big(\cap_{l=1}^a\cap_{b_a=1}^{k_a}\Big\{
Z_n({\bf p}(a,b_a))=T({\bf p}(a,b_a))\Big\}\Big)\\
-&\prod_{a=1}^l\Q_n\Big(\cap_{b_a=1}^{k_a}\Big\{Z_n({\bf
p}(a,b_a))=T({\bf p}(a,b_a))\Big\}\Big)\Big|,
\label{eq:cycle_bound_1}
\end{align}
where
\[
\Gamma:=\Big\{{\bf p}(a,b_a)\in\cU(n,a), b_a=1,2,\cdots, k_a,
a=1,2,\cdots,l\Big\}.
\]
Proceeding to analyze a generic term in the r.h.s. of \eqref
{eq:cycle_bound_1}, fix
\[
{\bf p}(a,b_a)\in\cU(n,a),\quad1\le b_a\le k_a, 1\le a\le l.
\]
Let $L_a=L_a\{{\bf p}(a,b_a),1\le b_a\le k_a\}$ denote the set of
distinct indices in the set $\{{\bf p}(a,b_a),1\,{\le}\break b_a\le k_a\}$. First
note that if the sets $L_a$ do not overlap across $a$, both terms in
the r.h.s. of \eqref{eq:cycle_bound_1} are the same, and so gets
canceled. As an example, this happens for the choice
\[
l=3,k_1=0,k_2=1,k_3=2, \quad {\bf p}(2,1)=(1,2),\quad {\bf
p}(3,1)=(3,4,5),\quad {\bf p}(3,2)=(3,4,5).
\]
In this case $L_1=\phi, L_2=\{1,2\}$ and $L_3=\{3,4,5\}$ do not
overlap, and so the corresponding terms in the r.h.s. of \eqref
{eq:cycle_bound_1} get cancelled.

If the sets $L_a$ do overlap across $a$, then the first term in the
r.h.s. of \eqref{eq:cycle_bound_1} is $0$. In this case setting $L:=\sum
_{a=1}^{l}{|L_a|}$ the total contribution of the second term in the
r.h.s. of \eqref{eq:cycle_bound_1} is bounded by
$\Big(\frac{M}{mn}\Big)^L$. Since there is a repetition among the
indices, the number of distinct indices $L(D)$ in the set $\{{\bf
p}(a,b_a),1\le b_a\le k_a,1\le a\le l\}$ is strictly less than $L$. As
an example, this happens for the choice
\[
l=3,k_1=0,k_2=1,k_3=2,\quad 
{\bf p}(2,1)=(1,2),\quad {\bf p}(3,1)=(3,5,4),\quad {\bf p}(3,2)=(1,6,7).
\]
In this case $L_1=\phi, L_2=\{1,2\}, L_3=\{1,3,4,5,6,7\}$, and so the
number of distinct indices $L(D)=7$ which is less than $L=|L_2|+|L_3|=8$.
Setting $K:=\sum_{a=1}^l k_a$, the total number of terms with exactly
$L(D)$ distinct indices is at most ${n\choose L(D)}K!$. Summing over
the possible ranges $L(D)\in[1,L-1], L\in[1,K]$ the total
contribution of such terms is at most
\[
\sum_{L=1}^K\sum_{L(D)=1}^{L-1}{n\choose L(D)}K!\Big(\frac{M}{mn}\Big
)^L=O\Big(\frac{1}{n}\Big),
\]
thus proving \eqref{eq:cycles_1}.

Proceeding to prove \eqref{eq:cycles_2} we again have
\begin{align}
\notag|\E\prod_{a=1}^lC_n(a)^{k_a}-\E\prod_{a=1}^lD_n(a)^{k_a}|\le
&\sum_{\Gamma}\Big|\P_n\Big(\cap_{a=1}^l\cap_{b_a=1}^{k_a}\pi_n({\bf
p}(a,b_a))=T({\bf p}(a,b_a))\Big)\\
-&\Q_n\Big(\cap_{a=1}^l\cap_{b_a=1}^{k_a}Z_n({\bf p}(a,b_a))=T({\bf
p}(a,b_a))\Big)\Big|
\label{eq:cycle_bound}
\end{align}
Proceeding to bound the r.h.s. of \eqref{eq:cycle_bound}, note that in
this case if all the indices in the set $\cup_{a=1}^l L_a$ are not
distinct (i.e. $L(D)\ne L$), then both terms in the r.h.s. of \eqref
{eq:cycle_bound} are 0. Even if $L(D)=L$, it is possible that both
terms are $0$, which happens for example for the choice
\[
l=3,k_1=0,k_2=1,k_3=2,\quad 
{\bf p}(2,1)=(1,2),\quad {\bf p}(3,1)=(3,5,4),\quad {\bf p}(3,2)=(3,4,5).
\]
In this case $L_1=\{1,2\}, L_2=\{3,4,5\}$ and so $L(D)=L=5$. However
both the terms on the r.h.s. of \eqref{eq:cycle_bound} have $0$ probability.
If either of the terms have non zero probability, then a generic term
on the r.h.s. of \eqref{eq:cycles_2} is of the form $| \P_n(\pi_n({\bf
p})={\bf q})-\Q_n(Z_n({\bf p})={\bf q})|$ for some ${\bf p},{\bf q}\in
\cS(n,l)$ with $l\in[L]$. Noting that $L\le K$, this can be bounded by
\[
\max_{l\in[K]}\max_{{\bf p},{\bf q}\in\cS(n,l)}\Big|\frac{\P_n(\pi
_n({\bf p})={\bf q})}{\Q_n(Z_n({\bf p})={\bf q})}-1\Big|\Q_n\Big(\cap
_{a=1}^l\cap_{b_a=1}^{k_a}Z_n({\bf p}(a,b_a))=T({\bf p}(a,b_a))\Big)\Big|
.
\]
On summing over $\Gamma$ using \eqref{eq:cycle_bound} gives
\[
|\E\prod_{a=1}^lC_n(a)^{k_a}-\E\prod_{a=1}^lD_n(a)^{k_a}|\le\max
_{l\in[K]}\max_{{\bf p},{\bf q}\in\cS(n,l)}\Big|\frac{\P_n(\pi_n({\bf
p})={\bf q})}{\Q_n(Z_n({\bf p})={\bf q})}-1\Big|\E\prod_{a=1}^lD_n(a)^{k_a},
\]
from which \eqref{eq:cycles_2} follows on using \eqref{eq:density_est}
along with \eqref{eq:cycles_1}.
\end{proof}

\section{Proof of Corollaries \ref{cor:mallows}-\ref{cor:mu_random} and
Proposition \ref{ppn}}\label{sec:three}

\begin{proof}[Proof of Corollary \ref{cor:mallows}]
By \cite[Theorem 1]{Starr} it follows that $\pi_n$ converges weakly in
probability to the measure $\mu_{\rho_\beta}$ induced by the density
$\rho_\beta$ defined in \eqref{eq:mallows_density}. Given Theorems \ref
{thm:density_est}, \ref{thm:fixed_points} and \ref{thm:cycle}, for
proving both parts (a) and (b) it suffices to verify the
equi-continuity condition \eqref{eq:equi_new}, which is equivalent to
the following two conditions:
\begin{align}\label{eq:equi_new1}
\lim_{\delta\rightarrow0}\lim_{n\rightarrow\infty}\sup_{{\bf p},{\bf
q},{\bf r}\in\cS(n,l):||{\bf p}-{\bf r}||_\infty\le n\delta}\Big|\frac
{\P_n(\pi_n({\bf p})={\bf q})}{\P_n(\pi_n({\bf r})={\bf q})}-1\Big|=0,\\
\label{eq:equi_new2} \lim_{\delta\rightarrow0}\lim_{n\rightarrow\infty
}\sup_{{\bf q},{\bf r},{\bf s}\in\cS(n,l):||{\bf q}-{\bf s}||_\infty
\le n\delta}\Big|\frac{\P_n(\pi_n({\bf r})={\bf q})}{\P_n(\pi_n({\bf
r})={\bf s})}-1\Big|=0.
\end{align}
Recall that \eqref{eq:equi_new1} was already verified in \cite
[Corollary 6.3+Lemma 7.1]{Bhattacharya-M}. By repeating the argument
presented there, we prove both \eqref{eq:equi_new1} and \eqref
{eq:equi_new2} here for completeness.
To show \eqref{eq:equi_new1}, fix ${\bf p}, {\bf q}, {\bf r}$ such that
$\norm{{\bf p}-{\bf r}}_\infty\le n\delta$.
Let $\Omega({\bf p},{\bf q})$ denote the set of all permutations in
$S_n$ such that $\pi_n({\bf p})={\bf q}$, and $\Omega({\bf r},{\bf q})$
be defined likewise. We will now define a bijection $\Phi=\Phi[({\bf
p},{\bf q});({\bf r},{\bf q})]$ from $\Omega({\bf p},{\bf q})$ to
$\Omega({\bf r},{\bf q}) $. For any $\pi_n\in\Omega({\bf p},{\bf q})$
set 
\[
\Phi(\pi_n)({\bf r})={\bf q},\quad \Phi(\pi_n)({\bf p}):=\pi_n({\bf
r}),\quad\Phi(\pi_n)(i)=\pi_n(i)\text{ otherwise}.
\]
It is easy to see that $\Phi$ is indeed a bijection, and
\begin{align*}
\frac{M_{n,q(n)}(\pi_n)}{M_{n,q(n)}(\Phi(\pi_n))}=q(n)^{Inv(\pi
_n)-Inv(\Phi(\pi_n))}\le\max\Big(q(n),q(n)^{-1}\Big)^{nl\delta},
\end{align*}
where we use the fact that the inversion status of a pair $(i,j)$ in
$\pi_n$ is the same as its inversion status in $\Phi(\pi_n)$ unless
$i\in\cup_{a=1}^l [p_a,r_a]$ and $j\in{\bf q}$. Summing over $\pi
_n\in\Omega({\bf p},{\bf q})$ gives
\[
\frac{\P_n(\pi_n({\bf p})={\bf q})}{\P_n(\pi_n({\bf r})={\bf q})}\le
\max\Big(q(n),q(n)^{-1}\Big)^{nl\delta},
\]
and since the bound in the r.h.s. above is free of ${\bf p},{\bf
q},{\bf r}$, taking a sup gives
\[
\sup_{{\bf p},{\bf q},{\bf r}\in\cS(n,l):||{\bf p}-{\bf r}||_\infty\le
n\delta}\frac{\P_n(\pi_n({\bf p})={\bf q})}{\P_n(\pi_n({\bf r})={\bf
q})}\le \max\Big(q(n),q(n)^{-1}\Big)^{nl\delta}.
\]
On letting $n\rightarrow\infty$ followed by $\delta\rightarrow0$ and
noting that $n(1-q(n))\rightarrow\beta\in(-\infty,\infty)$, we get
\[
\limsup_{\delta\rightarrow0}\limsup_{n\rightarrow\infty}\sup_{{\bf
p},{\bf q},{\bf r}\in\cS(n,l):||{\bf p}-{\bf r}||_\infty\le n\delta
}\frac{\P_n(\pi_n({\bf p})={\bf q})}{\P_n(\pi_n({\bf r})={\bf q})}\le1,
\]
thus giving the upper bound in \eqref{eq:equi_new1}.
By symmetry we have
\[
\liminf_{\delta\rightarrow0}\liminf_{n\rightarrow\infty}\sup_{{\bf
p},{\bf q},{\bf r}\in\cS(n,l):||{\bf p}-{\bf r}||_\infty\le n\delta
}\frac{\P_n(\pi_n({\bf p})={\bf q})}{\P_n(\pi_n({\bf r})={\bf q})}\ge1,
\]
thus giving the lower bound, and hence proving \eqref{eq:equi_new1}.
For proving \eqref{eq:equi_new2} a similar argument works, except now
we set up the bijection $\widetilde{\Phi}_n=\widetilde{\Phi}_n[({\bf
r},{\bf q});({\bf r},{\bf s})]$ between $\Omega_{{\bf r},{\bf q}}$ to
$\Omega_{{\bf r},{\bf s}}$ by setting
\[
\widetilde{\Phi}(\pi_n)({\bf r})={\bf s},\quad \widetilde{\Phi}(\pi
_n)(\pi_n^{-1}{\bf s}):={\bf q},\quad\Phi(\pi_n)(i)=\pi_n(i)\text{ otherwise}.
\]
The rest of the argument repeats itself, and we omit the details.
\end{proof}

\begin{proof}[Proof of Corollary \ref{cor:new_limit}]

\begin{enumerate}
\item[(a)]
It follows from \cite[Theorem 1.4]{Mukherjee} that
$\pi_n$ converges to a unique measure $\mu_{f,\theta}$ weakly in
probability, which is the solution of the optimization problem
\[
\mu\mapsto\{\theta\mu[f]-D(\mu||u)\},
\]
where $u$ is the uniform measure on the unit square, and $D(.||.)$ is
the Kullback Leibler divergence. It was further shown there that $\mu
_{f,\theta}$ has a density of the form $g_{f,\theta}(x,y)=e^{\theta
f(x,y)+a_{f,\theta}(x)+b_{f,\theta}(y)}$, where $a_{f,\theta}(.)$ and
$b_{f,\theta}(.)$ are unique almost surely. To complete the proof of
part (a), it suffices to show that the function $g_{f,\theta}$ is
continuous on the unit square, or equivalently that $e^{-a_{f,\theta
}(.)}$ is continuous. To this effect, using the fact that $\mu_{f,\theta
}$ has uniform marginals, we have
\[
e^{-a_{f,\theta}(x)}=\int_0^1 e^{\theta f(x,y)+b_{f,\theta}(y)}dy,
\]
which readily gives
\[
\int_0^1 e^{b_{f,\theta}(y)}dy\le e^{-a_{f,\theta}(x)-\inf_{x,y\in
[0,1]}\{\theta f(x,y)\}}
\]
for almost all $x\in[0,1]$, and consequently $e^{b_{f,\theta}(.)}$ is
integrable. But then we have
\[
\Big|e^{-a_{f,\theta}(x_1)}-e^{-a_{f,\theta}(x_2)}\Big|\le\sup_{y\in
[0,1]}\Big|e^{\theta f(x_1,y)}-e^{\theta f(x_2,y)}\Big|\int_0^1
e^{b_{f,\theta}(y)}dy,
\]
from which continuity of $e^{-a_{f,\theta}(.)}$ follows from continuity
of $f(.,.)$.

\item[(b),(c)]
As in the proof of Corollary \ref{cor:mallows} it suffices to verify
the conditions \eqref{eq:equi_new1} and \eqref{eq:equi_new2}. Using the
same notations as in the proof of Corollary \ref{cor:mallows}, we have
\begin{align*}
\frac{\Q_{n,f,\theta}(\pi_n)}{\Q_{n,f,\theta}(\Phi(\pi_n))}=e^{\theta
\sum_{a=1}^l f(p_a/n,q_a/n)-f(r_a/n,q_a/n)},
\end{align*}
and the exponent in the r.h.s. above is bounded by
\[
|\theta| \sup_{x_1,x_2,y\in[0,1]:|x_1-x_2|\le\delta}|f(x_1,y)-f(x_2,y)|.
\]
Since this goes to $0$ as $\delta\rightarrow0$, a similar proof as
before verifies \eqref{eq:equi_new1}. The proof of \eqref{eq:equi_new2}
is similar, and again we omit the details.\qedhere
\end{enumerate}
\end{proof}

\begin{proof}[Proof of Corollary \ref{cor:mu_random}]
Since a sequence of $\mu_\rho$ random permutations converge to $\mu_\rho
$ weakly in probability, it suffices to verify \eqref{eq:equi_new1} and
\eqref{eq:equi_new2}.

To this effect, with $(X_1,Y_1),\cdots, (X_n,Y_n)\stackrel{i.i.d.}{\sim
}\mu_\rho$ first note that marginally both $(X_1,\cdots,X_n)$ and
$(Y_1,\cdots,Y_n)$ are i.i.d. $U(0,1)$. Thus if $(U_1,\cdots,U_n)$ and
$(V_1,\cdots,V_n)$ are the order statistics of $(X_1,\cdots,X_n)$ and
$(Y_1,\cdots,Y_n)$ respectively, for any $\delta>0$ we have
\begin{align}\label{eq:union}
\P_n\Big(\Big|U_i-\frac{i}{n}\Big|>\delta\Big)=\P_n\left(Bin\Big(n,\frac
{i}{n}-\delta\Big)\ge i\right)+\P_n\left(Bin\Big(n,\frac{i}{n}+\delta
\Big)\le i\right)\le2e^{-\delta^2 n}
\end{align}
by Hoeffding's inequality.
Also using \eqref{eq:mu_random}, for any ${\bf p},{\bf q}\in\cS(n,l)$
we have
\begin{align*}
\P_n(\pi_n({\bf p})={\bf q})=
&n!\sum_{\pi_n\in\Omega({\bf p},{\bf q})}\int_{u_1<u_2<\cdots
,u_n,v_1<v_2<\cdots<v_n} \prod_{i=1}^n f\Big(u_i,v_{\pi_n(i)}\Big) du_i dv_i,
\end{align*}
which, for ${\bf r}\in\cS(n,l)$ gives
\begin{align*}
\frac{\P_n(\pi_n({\bf p})={\bf q})}{\P_n(\pi_n({\bf p})={\bf r})}\le
\sup_{{\bf x},{\bf y },{\bf z}\in[0,1]^l}\prod_{a=1}^l\frac{\rho
(x_a,y_a)}{\rho(x_a,z_a)}\le\Big(\frac{M}{m}\Big)^l.
\end{align*}
Noting that $|\cS(n,l)|\le n^l$, summing over ${\bf r}$ this implies
\begin{align}\label{eq:poly_estimate}
\P_n(\pi_n({\bf p})={\bf q})\ge\Big(\frac{m}{Mn}\Big)^l\sum_{{\bf
r}\in\cS(n,l)}\P_n(\pi_n({\bf p})={\bf r})= \Big(\frac{m}{nM}\Big)^l
\end{align}
Finally, for any ${\bf p}, {\bf q},{\bf r}\in\cS(n,l)$ such that
$||{\bf p}-{\bf r}||_\infty\le n\delta$, setting $A_n:=\{\max_{i\in
[n]}\Big|U_i-\frac{i}{n}\Big|\le\delta\}$
we have
\begin{align}
\notag\frac{\P_n(\pi_n({\bf p})={\bf q})}{\P_n(\pi_n({\bf r})={\bf
q})}\le& \frac{\P_n(A_n^c)+\P_n(\pi_n({\bf p})={\bf q}, A_n)}{\P_n(\pi
_n({\bf r})={\bf q}, A_n)}\\
\notag=&\frac{\P_n(A_n^c)+n!\sum\limits_{\pi_n\in\Omega({\bf p},{\bf
q})}\int_{u_1<u_2<\cdots< u_n,v_1<v_2<\cdots,<v_n,A_n}\prod_{i=1}^n \rho
\Big(u_i,v_{\pi_n(i)}\Big)}{n!\sum\limits_{\pi_n\in\Omega({\bf r},{\bf
q})}\int_{u_1<u_2<\cdots< u_n,v_1<v_2<\cdots,<v_n,A_n}\prod_{i=1}^n \rho
\Big(u_i,v_{\pi_n(i)}\Big)}\\
\label{eq:murandom1}\le&\max_{{\bf p},{\bf q},{\bf r}\in\cS
(n,l):||{\bf p}-{\bf r}||_\infty\le n\delta}\sup_{{\bf u},{\bf v}\in
[0,1]^n:|u_i-\frac{i}{n}|\le\delta}\frac{\prod_{a=1}^l \rho\Big
(u_{p_a},v_{q_a}\Big)}{\prod_{a=1}^l \rho\Big(u_{r_a},v_{q_a}\Big)}\\
\label{eq:murandom2}+&\frac{\P_n(A_n^c)}{\P_n(\pi_n({\bf r})={\bf q})-\P
_n(A_n^c)}.
\end{align}
Since
$|u_{p_a}-u_{r_a}|\le2\delta+\frac{|p_a-r_a|}{n}\le3\delta,$
the expression in \eqref{eq:murandom1} can be bounded by
\[
\sup_{{\bf x},{\bf y},{\bf z}\in[0,1]^l:||{\bf x}-{\bf z}||_\infty\le
3\delta}\frac{\prod_{a=1}^l \rho(x_a,y_a)}{\prod_{a=1}^l \rho(z_a,y_a)}
\]
which is free of $n$, and goes to $0$ as $\delta\rightarrow0$ by
continuity of $\rho$. Also, using \eqref{eq:union} and \eqref
{eq:poly_estimate} it follows that the expression in \eqref
{eq:murandom2} is bounded above by
\[
\frac{2e^{-n\delta^2}}{\Big(\frac{M}{nm}\Big)^l-2e^{-n\delta^2}},
\]
which converges to $0$ as $n\rightarrow\infty$, for every $\delta$
fixed. Thus, taking a maximum over
${\bf p},{\bf q},{\bf r}\in\cS(n,l)$ such that $||{\bf p}-{\bf
r}||_\infty\le n\delta$ we have
\[
\limsup_{\delta\rightarrow0}\limsup_{n\rightarrow\infty}\max_{{\bf
p},{\bf q},{\bf r}\in\cS(n,l),||{\bf p}-{\bf r}||_\infty\le n\delta
}\frac{\P_n(\pi_n({\bf p})={\bf q})}{\P_n(\pi_n({\bf r})={\bf q})}\le1,
\]
thus giving the upper bound in \eqref{eq:equi_new1}. Similar arguments
give the lower bound in \eqref{eq:equi_new1}, as well as \eqref
{eq:equi_new2}, thus completing the proof of the corollary.
\end{proof}

\begin{proof}[Proof of Proposition \ref{ppn}]
With $\P_n=\R_{n,0}$ denoting the uniform measure on $S_n$ and $D_n$
denoting the number of derangements of $n$, we have
\begin{align}\label{eq:derangements}
\frac{1}{n!}e^{Z_n(\theta)}=\E_{\P_n} e^{\theta N_n(\pi_n,e_n)}=\sum
_{k=0}^\infty e^{\theta k} \frac{{n\choose k}D_{n-k}}{n!}\rightarrow
\text{exp}\{e^\theta-1\},
\end{align}
where we use the fact that $D_n/n!$ converges to $e^{-1}$.
\begin{enumerate}
\item[(a)]
For any $\lambda>0$ we have
\begin{align*}
\E_{\R_{n,\theta}} e^{\lambda N_n(\pi_n,e_n)}=e^{Z_n(\theta+\lambda
)-Z_n(\theta)}\rightarrow\text{ exp}\{ e^\theta(e^\lambda-1)\},
\end{align*}
and so $N_n(\pi_n,e_n)$ converges to $Poi(e^\theta)$ in distribution
and in moments.

\item[(b)]

With $D(.||.)$ denoting the Kullback-Leibler divergence we have
\[
D(\R_{n,0}||\R_{n,\theta})= \log\Big(\frac{e^{Z_n(\theta)}}{n!}\Big
)-\theta\E_{\P_n}N(\pi_n,e_n) \rightarrow e^\theta-1-\theta,
\]
and so by \cite[Prop 5.1]{BM} we have that the two probability
distributions $\R_{n,\theta}$ and $\R_{n,0}=\P_n$ are mutually
contiguous. Since $\pi_n$ converges weakly to $u$ under $\P_n=\R
_{n,0}$, by contiguity the same happens for $\R_{n,\theta}$.

\item[(c)]

Let $A_n:=\{\pi_n\in S_n:\pi_n(1)=1,\pi_n(2)=2\}$, and $B_n:=\{\pi_n\in
S_n:\pi_n(1)=2,\pi_n(2)=1\}$. Define a bijection $\omega$ from $A_n$ to
$B_n$ by setting $\omega(\pi_n)(i)=i$ for $3\le i\le n$, and note that
\[
\frac{\R_{n,\theta}(\pi_n)}{\R_{n,\theta}(\omega(\pi_n))}=e^{2\theta},
\]
and so summing over $\pi_n\in A_n$ gives
\[
\frac{\R_{n,\theta}(\pi_n(1)=1,\pi_n(2)=2)}{\R_{n,\theta}(\pi_n(1)=2,\pi
_n(2)=1)}=e^{2\theta}\ne1,
\]
thus proving part (c).\qedhere
\end{enumerate}
\end{proof}

\section{Acknowledgements}

The Poisson distribution for the number of fixed points in the Mallows
model with Kendall's Tau was conjectured by Susan Holmes based on
empirical evidence. This paper also benefited from helpful discussions
with Shannon Starr. Suggestions from an anonymous referee greatly
improved the presentation of the paper.


\begin{thebibliography}{99}
\bibitem{AGG}
{\sc Arratia, R.}, {\sc Goldstein, L.}, {\sc and} {\sc Gordon, L.} (1990).
\newblock Poisson approximation and the {C}hen-{S}tein method.
\newblock{\em Statist. Sci.\/}~\textbf{5},~4, 403--434.
\newblock With comments and a rejoinder by the authors.
\MR{1092983}

\bibitem{BB}
{\sc Basu, R.} {\sc and} {\sc Bhatnagar, N.} (2016).
\newblock Limit Theorems for Longest Monotone Subsequences in Random
Mallows Permutations. \newblock Available at \url
{http://arxiv.org/pdf/1601.02003}.

\bibitem{Bhattacharya-M}
{\sc Bhattachara, B.} {\sc and} {\sc Mukherjee, S.} (2015).
\newblock Degree sequence of random permutation graphs.
\newblock{\em Ann. Appl. Probab.}, to appear.

\bibitem{BM}
{\sc Bhattachara, B.} {\sc and} {\sc Mukherjee, S.} (2015).
\newblock Inference in Ising models.
\newblock Available at \url{http://arxiv.org/abs/1507.07055}.

\bibitem{BP}
{\sc Bhatnagar, N.} {\sc and} {\sc Peled, R.} (2015).
\newblock Lengths of monotone subsequences in a {M}allows permutation.
\newblock{\em Probab. Theory Related Fields\/}~\textbf{161},~3--4, 719--780.
\MR{3334280}

\bibitem{BDF}
{\sc Borodin, A.}, {\sc Diaconis, P.}, {\sc and} {\sc Fulman, J.} (2010).
\newblock On adding a list of numbers (and other one-dependent determinantal
processes).
\newblock{\em Bull. Amer. Math. Soc. (N.S.)\/}~\textbf{47},~4, 639--670.
\MR{2721041}


\bibitem{CDM}
{\sc Chatterjee, S.}, {\sc Diaconis, P.}, {\sc and} {\sc Meckes, E.} (2005).
\newblock Exchangeable pairs and {P}oisson approximation.
\newblock{\em Probab. Surv.\/}~{\em2}, 64--106.
\MR{2121796}

\bibitem{Diaconis}
{\sc Diaconis, P.} (1988).
\newblock{\em Group representations in probability and statistics}.
\newblock Institute of Mathematical Statistics Lecture Notes---Monograph
Series, 11. Institute of Mathematical Statistics, Hayward, CA.
\MR{0964069}

\bibitem{Diaconis-Ram}
{\sc Diaconis, P.} {\sc and} {\sc Ram, A.} (2000)
\newblock Analysis of Systematic Scan Metropolis Algorithms Using
Iwahori-Hecke Algebra Techniques.
\newblock{\em Michigan Math. J.\/}~\textbf{48},~1, 157--190.
\MR{1786485}

\bibitem{GP}
{\sc Gladkich, A.} {\sc and} {\sc Peled, R.} (2016)
\newblock On the cycle structure of Mallows permutations.
\newblock Available at \url{http://arxiv.org/pdf/1601.06991}.

\bibitem{HKMRS}
{\sc Hoppen, C.}, {\sc Kohayakawa, Y.}, {\sc Moreira, C.~G.}, {\sc R{\'a}th,
B.}, {\sc and} {\sc Menezes~Sampaio, R.} (2013).
\newblock Limits of permutation sequences.
\newblock{\em J. Combin. Theory Ser. B\/}~\textbf{103},~1, 93--113.
\MR{2995721}

\bibitem{KKRW}
{\sc Kenyon, R.}, {\sc Kr\'al, D.}, {\sc Radin, C.}, {\sc and } {\sc
Winkler, P.}(2015).
\newblock A variational principle for permutations.
\newblock Available at \url{http://arxiv.org/pdf/1506.02340}.

\bibitem{Mallows}
{\sc Mallows, C.~L.} (1957).
\newblock Non-null ranking models. {I}.
\newblock{\em Biometrika\/}~{\em44}, 114--130.
\MR{0087267}

\bibitem{MS}
{\sc Mueller, C.} {\sc and} {\sc Starr, S.} (2013).
\newblock The length of the longest increasing subsequence of a random
{M}allows permutation.
\newblock{\em J. Theoret. Probab.\/}~\textbf{26},~2, 514--540.
\MR{3055815}

\bibitem{Mukherjee}
{\sc Mukherjee, S.} (2016).
\newblock Estimation in exponential families on permutations.
\newblock{\em Ann. Statist.\/}~\textbf{44},~2, 853--875.
\MR{3476619}

\bibitem{Nelson}
{\sc Nelsen, R.~B.} (2006).
\newblock{\em An introduction to copulas\/}, Second ed.
\newblock Springer Series in Statistics. Springer, New York.
\MR{2197664}

\bibitem{Starr}
{\sc Starr, S.} (2009).
\newblock Thermodynamic limit for the {M}allows model on {$S\sb n$}.
\newblock{\em J. Math. Phys.\/}~\textbf{50},~9, 095208, 15.
\MR{2566888}

\bibitem{SW}
{\sc Walters, M.} {\sc and} {\sc Starr, S.} (2015).
\newblock A note on mixed matrix moments for the complex {G}inibre ensemble.
\newblock{\em J. Math. Phys.\/}~\textbf{56},~1, 013301, 20.
\MR{3390837}

\bibitem{J}
{\sc Trashorras, J.} (2008).
\newblock Large deviations for symmetrised empirical measures.
\newblock{\em J. Theoret. Probab.\/}~\textbf{21},~2, 397--412.
\MR{2391251}

\end{thebibliography}
\end{document}